\documentclass{amsart}
\usepackage{amsfonts}

\newtheorem{theorem}{Theorem}
\theoremstyle{plain}

\newtheorem{corollary}{Corollary}

\newtheorem{lemma}{Lemma}

\newtheorem{proposition}{Proposition}
\newtheorem{remark}{Remark}

\numberwithin{equation}{section}
\input{tcilatex}

\begin{document}
\title[Zeta functions]{A few remarks on values of Hurwitz Zeta function at
natural and rational arguments}
\author{Pawe\l\ J. Szab\l owski}
\address{Department of Mathematics and Information Sciences,\\
Warsaw University of Technology\\
ul. Koszykowa 75, 00-662 Warsaw, Poland}
\email{pawel.szablowski@gmail.com}
\date{April 2014}

\begin{abstract}
We exploit some properties of the Hurwitz zeta function $\zeta (n,x)$ in
order to study sums of the form $\frac{1}{\pi ^{n}}\sum_{j=-\infty }^{\infty
}1/(jk+l)^{n}$ and \newline
$\frac{1}{\pi ^{n}}\sum_{j=-\infty }^{\infty }(-1)^{j}/(jk+l)^{n}$ for $%
2\leq n,k\in \mathbb{N},$ and integer $l\leq k/2$. We show that these sums
are algebraic numbers. We also show that $1<n\in \mathbb{N}$ and $p\in 
\mathbb{Q\cap (}0,1\mathbb{)}$ $:$ the numbers $(\zeta (n,p)+(-1)^{n}\zeta
(n,1-p))/\pi ^{n}$ are algebraic. On the way we find polynomials $s_{m}$ and 
$c_{m}$ of order respectively $2m+1$ and $2m+2$ such that their $n-$th
coefficients of sine and cosine Fourier transforms are equal to $%
(-1)^{n}/n^{2m+1}$ and $(-1)^{n}/n^{2m+2}$ respectively.
\end{abstract}

\maketitle

\section{Introduction}

Firstly we find polynomials $s_{m}$ and $c_{m}$ of orders respectively $2m+1$
and $2m+2$ such that their $n-$th coefficients in respectively sine and
cosine Fourier series are of the form $(-1)^{n}/n^{2m+1}$ and $%
(-1)^{n}/n^{2m+2}.$ Secondly using these polynomials we study sums of the
form: 
\begin{equation*}
S(n,k,l)\allowbreak =\allowbreak \sum_{j=-\infty }^{\infty }\frac{1}{%
(jk+l)^{n}},~\hat{S}(n,k,l)=\sum_{j=-\infty }^{\infty }\frac{(-1)^{j}}{%
(jk+l)^{n}}.
\end{equation*}%
for $n,$ $k,$ $l\in \mathbb{N}$ such that $n\geq 2,$ $k\geq 2$ generalizing
known result of Euler (recently proved differently by Elkies (\cite{Elk03}))
stating that the number $S(n,4,1)/\pi ^{n}\allowbreak $ is rational. In
particular we show that the numbers $S(n,k,j)/\pi ^{n}$ and $\hat{S}%
(n,k,j)/\pi ^{n}$ for $\mathbb{N\ni }n,k\geq 2,$ and $1\leq j<k$ are
algebraic.

The results of the paper are somewhat in the spirit of \cite{cvij+kal99} in
the sense that the sums we study are equal to certain combinations of values
of the Hurwitz zeta function calculated for natural and rational arguments.

The paper is organized as follows. In the present section we recall basic
definitions and a few elementary properties of the Hurwitz and Riemann zeta
functions. We present them for completeness of the paper and also to
introduce in this way our notation.

In Section \ref{main} are our main results. Longer proof are shifted to
Section \ref{dowody}.

To fix notation let us recall that the following function

\begin{equation*}
\zeta (s,x)\allowbreak =\allowbreak \sum_{n\geq 0}1/(n+x)^{s},
\end{equation*}%
defined for all $\func{Re}(s)>1$ and $\ 1\geq \func{Re}(x)>0$ is called
Hurwitz zeta function while the function $\zeta (s,1)\overset{df}{=}\zeta
(s) $ is called the Riemann zeta function. One shows that functions $\zeta
(s,x)$ and $\zeta (s)$ can be extended to meromorphic functions of $s$
defined for all $s\neq 1$.

Let us consider slight generalization of the Hurwitz zeta function namely
the function;%
\begin{equation*}
\hat{\zeta}(s,x)=\allowbreak \sum_{n\geq 0}(-1)^{n}/(n+x)^{s},
\end{equation*}%
for the same as before $s$ and $x.$

Using multiplication theorem for $\zeta (s,x)$ and little algebra we get:%
\begin{equation}
\zeta (s,x)\allowbreak =\allowbreak (\zeta (s,x/2)+\zeta
(s,(1+x)/2))/2^{s},~~\hat{\zeta}(s,x)=(\zeta (s,x/2)-\zeta
(s,(1+x)/2))/2^{s},  \label{HZ}
\end{equation}%
\begin{equation}
S(n,k,l)\allowbreak =\allowbreak \frac{1}{k^{n}}(\zeta (n,\frac{l}{k}%
)+(-1)^{n}\zeta (n,1-\frac{l}{k})),~\hat{S}(n,k,l)\allowbreak =\allowbreak 
\frac{1}{k^{n}}(\hat{\zeta}(n,\frac{l}{k})+(-1)^{n+1}\hat{\zeta}(n,1-\frac{l%
}{k})),  \label{SHZ}
\end{equation}
\begin{gather}
S(n,k,l)\allowbreak =\allowbreak S(n,2k,l)+(-1)^{n}S(n,2k,k-l),  \label{S1_2}
\\
\hat{S}(n,k,l)\allowbreak =\allowbreak S(n,2k,l)+(-1)^{n+1}S(n,2k,k-l),
\label{S*} \\
S(n,k,k-l)\allowbreak =\allowbreak (-1)^{n}S(n,k,l),~\hat{S}%
(n,k,k-l)=(-1)^{n+1}\hat{S}(n,k,l).  \label{S**}
\end{gather}

Hence there is sense to define sums $S$ and $\hat{S}$ for $l\leq k/2$ only.

In the sequel $\mathbb{N}$ and $\mathbb{N}_{0}$ will denote respectively
sets of natural and natural plus $\left\{ 0\right\} $ numbers while $\mathbb{%
Z}$ and $\mathbb{Q}$ will denote respectively sets of integer and rational
numbers. $B_{k}$ will denote $k-th$ Bernoulli number. It is known that $%
B_{2n+1}\allowbreak =\allowbreak 0$ for $n\geq 1$ and $B_{1}\allowbreak
=\allowbreak 1/2.$ One knows also that 
\begin{equation}
\zeta (2l)\allowbreak =\allowbreak (-1)^{l+1}B_{2l}\frac{(2\pi )^{2l}}{2(2l)!%
}\allowbreak  \label{zeta_parz}
\end{equation}%
for all $l\in \mathbb{N}$.

For more properties of the Hurwitz zeta function and the Bernoulli numbers
see \cite{Kara-Voronin92}, \cite{Chan53}, \cite{Magnus66} and \cite{Car68}.

Having this and (\ref{S1_2}) and applying multiplication theorem for $\zeta
(s,x)$ we can easily generalize result of Euler reminded by Elkies in \cite%
{Elk03}). Let us recall that since Euler times it is known that. 
\begin{equation*}
\frac{1}{\pi ^{n}}S(n,4,1)=\left\{ 
\begin{array}{ccc}
\frac{(2^{2l}-1)}{2(2l)!}(-1)^{l+1}B_{2l} & if & n=2l, \\ 
\frac{(-1)^{l}E_{2l}}{2^{2l+2}(2l)!} & if & n=2l+1.%
\end{array}%
\right.
\end{equation*}%
Apart from this result one can easily show (manipulating multiplication
theorem for Hurwitz zeta function) that for for $\in \mathbb{N}$: 
\begin{eqnarray*}
\frac{(2l)!}{\pi ^{2l}}S(2l,2,1)\allowbreak &=&\allowbreak \frac{2(2l)!}{%
(2\pi )^{2l}}\zeta (2l,\frac{1}{2})\allowbreak =\allowbreak
(2^{2l}-1)(-1)^{l+1}B_{2l},~ \\
\frac{2(2l)!}{\pi ^{2l}}S(2l,3,1)\allowbreak &=&\allowbreak \allowbreak (%
\frac{2}{3})^{2l}(-1)^{l+1}(3^{2l}-1)B_{2l},~ \\
\frac{2(2l)!}{\pi ^{2l}}S(2l,6,1)\allowbreak &=&\allowbreak
(2^{2l}-1)(1-3^{-2l})(-1)^{l+1}B_{2l}.
\end{eqnarray*}%
Notice also that from (\ref{HZ}) it follows that $S(2l+1,2,1)\allowbreak
=\allowbreak 0$ for $l\in \mathbb{N}$\textbf{\ }and from (\ref{HZ}) and (\ref%
{SHZ}) we get: $\hat{S}(2l,2,1)\allowbreak =\allowbreak 0$ and $\hat{S}%
(2l+1,2,1)\allowbreak =\allowbreak 2S(2l+1,4,1).$ Results concerning $%
S(2l,6,1)$ were also obtained in \cite{cvij+kal99}(16b).

\section{Main results\label{main}}

To get relationships between particular values of $\zeta (n,q)$ for $n\in 
\mathbb{N}$ and $q\in \mathbb{Q}$ we will use some elementary properties of
the Fourier transforms.

Let us denote 
\begin{equation*}
\mathcal{F}_{n}^{s}(f(.))\allowbreak =\mathcal{F}_{n}^{s}=\allowbreak \frac{1%
}{\pi }\int_{-\pi }^{\pi }f(x)\sin (nx)dx,~\mathcal{F}_{n}^{c}(f(.))%
\allowbreak \allowbreak =\mathcal{F}_{n}^{c}\allowbreak =\allowbreak \frac{1%
}{\pi }\int_{-\pi }^{\pi }f(x)\cos (nx)dx,
\end{equation*}%
for $n\in \mathbb{N}_{0}$. We will abbreviate $\mathcal{F}%
_{n}^{s}(f(.))\allowbreak $ and $\mathcal{F}_{n}^{c}(f(.))$ to $\mathcal{F}%
_{n}^{s}$ and $\mathcal{F}_{n}^{c}$ if the function $f$ is given. We have: 
\begin{equation*}
f(x)\allowbreak =\allowbreak \frac{1}{2}\mathcal{F}_{0}^{c}+\sum_{n\geq 1}%
\mathcal{F}_{n}^{c}\cos (nx)+\sum_{n\geq 1}\mathcal{F}_{n}^{s}\sin (nx)
\end{equation*}%
for all points of continuity of $f$ and $\left( f(x^{+})+f\left(
x^{-}\right) \right) /2$ if at $x$ $f$ has jump discontinuity as a result of
well known properties of Fourier series.

Let us define two families of polynomials: 
\begin{gather}
c_{m}(x)\allowbreak =\allowbreak \frac{x^{2}(-1)^{m+1}}{(2m+2)!}%
\sum_{j=0}^{m}\binom{2m+2}{2j}x^{2(m-j)}\pi ^{2j}B_{2j}(2^{2j-1}-1),
\label{wp} \\
s_{m}(x)=\frac{x\left( -1\right) ^{m}}{(2m+1)!}\sum_{j=0}^{m}\binom{2m+1}{2j}%
x^{2(m-j)}\pi ^{2j}B_{2j}(2^{2j-1}-1).  \label{wn}
\end{gather}%
for all $m\in \mathbb{N}_{0}.$

We have the following auxiliary lemma.

\begin{lemma}
\label{fourier} For $n\in \mathbb{N},$ $m\in \mathbb{N}_{0}$ 
\begin{gather}
\mathcal{F}_{n}^{s}(s_{m}(.))\allowbreak =\allowbreak \frac{\left( -1\right)
^{n}}{n^{2m+1}},~~\mathcal{F}_{n}^{c}(c_{m}(.))\allowbreak =\allowbreak 
\frac{\left( -1\right) ^{n}}{n^{2m+2}}.  \label{sniep} \\
\mathcal{F}_{0}^{c}(c_{m}(.))\allowbreak =\allowbreak \pi
^{2m+2}(-1)^{m}B_{2m+2}\frac{2(2^{2m+1}-1)}{(2m+2)!}\allowbreak =\allowbreak
2\zeta (2m+2)(1-2^{-2m-1})  \notag
\end{gather}
\end{lemma}

\begin{proof}
Is shifted to Section \ref{dowody}.
\end{proof}

\begin{remark}
Notice that we have just summed the two following Dirichlet series: $%
DC(s,x)=\sum_{j\geq 1}(-1)^{j}\frac{\cos (jx)}{j^{s}}$ and $%
DS(s,x)=\sum_{j\geq 1}(-1)^{j}\frac{\sin (jx)}{j^{s}}$, for particular
values of $s.$ Namely we have showed that 
\begin{eqnarray*}
DC(2m+2,x)\allowbreak &=&\allowbreak c_{m}(x)\allowbreak +\allowbreak \pi
^{2m+2}(-1)^{m+1}B_{2m+2}\frac{(2^{2m+1}-1)}{(2m+2)!}, \\
DS(2m+3,x)\allowbreak &=&\allowbreak s_{m+1}(x)
\end{eqnarray*}
and for $\left\vert x\right\vert \leq \pi ,$ $m\in \mathbb{N}_{0}$. What are
the values of these series for $s\neq 1,2,3,\ldots $?
\end{remark}

\begin{remark}
\label{ratio}Notice that $c_{m}(\frac{\pi l}{k})/\pi ^{2m+2}$ and $s_{m}(%
\frac{\pi l}{k})/\pi ^{2m+1}\allowbreak $ for all $k\in \mathbb{N},$ $m\in 
\mathbb{N}_{0}$ and $l\leq k$, are rational numbers.
\end{remark}

As a consequence we get the following theorem that allows calculating values
of $S(n,k,l)$ and $\hat{S}(n,k,l)$ for $\mathbb{N\ni }n,k\geq 2$ and $l\leq
k $.

\begin{theorem}
\label{parz}For $m\in \mathbb{N}_{0}$, $i\in \mathbb{N},\allowbreak $

i) 
\begin{gather}
c_{m}(\frac{(2l-1)\pi }{2i+1})\allowbreak =\allowbreak \frac{%
(-1)^{m}B_{2m+2}\pi ^{2m+2}}{(2m+2)!}(2^{2m+1}(1+\frac{1}{(2i+1)^{2m+2}})-1)
\label{cn} \\
+\allowbreak \sum_{j=1}^{i}(-1)^{j}\cos (\frac{j(2l-1)\pi }{2i+1}%
)S(2m+2,2i+1,j),  \notag
\end{gather}%
for $l\allowbreak =\allowbreak 1,\ldots ,i-1$.

ii) 
\begin{gather}
c_{m}(\frac{(2l-1)\pi }{2i})\allowbreak =\allowbreak \frac{\pi
^{2m+2}(-1)^{m}B_{2m+2}(2^{2m+1}-1)}{(2m+2)!}(1-(2i)^{-2m-2})\allowbreak
\label{cp} \\
+\allowbreak \sum_{j=1}^{i-1}(-1)^{j}\cos (\frac{(2l-1)j\pi }{2i})\hat{S}%
(2m+2,2i,j)\allowbreak ,  \notag
\end{gather}

for $l\allowbreak =\allowbreak 1,\ldots ,i$.

iii) 
\begin{gather}
c_{m}(\frac{2l\pi }{2i+1})\allowbreak =\allowbreak \frac{\pi
^{2m+2}(-1)^{m}B_{2m+2}(2^{2m+1}-1)}{(2m+2)!}(1-(2i+1)^{-2m-2})\allowbreak
\label{cnp} \\
+\allowbreak \sum_{j=1}^{i}(-1)^{j}\cos (\frac{2lj\pi }{2i+1})\hat{S}%
(2m+2,2i+1,j)\allowbreak ,  \notag
\end{gather}

for $l\allowbreak =\allowbreak 1,\ldots ,i$.

iv) For $m,k\in \mathbb{N}$: 
\begin{equation}
s_{m}(\frac{(2l-1)\pi }{2i+1})\allowbreak =\allowbreak
\sum_{j=1}^{i}(-1)^{j}\sin (\frac{(2l-1)j\pi }{2i+1})S(2m+1,2i+1,j),
\label{sn}
\end{equation}%
for $l\allowbreak =\allowbreak 1,\ldots i-1$.

v) 
\begin{gather}
s_{m}(\frac{(2l-1)\pi }{2i})\allowbreak =\allowbreak (-1)^{i+l+1}\frac{1}{%
i^{2m+1}}S(2m+1,4,1)\allowbreak  \label{sp} \\
+\allowbreak \sum_{j=1}^{i-1}(-1)^{j}\sin (\frac{(2l-1)j\pi }{2i})\hat{S}%
(2m+1,2i,j),  \notag
\end{gather}%
for $l\allowbreak =\allowbreak 1,\ldots i$.

vi) 
\begin{equation}
s_{m}(\frac{2l\pi }{2i+1})\allowbreak =\allowbreak
\sum_{j=1}^{i}(-1)^{j}\sin (\frac{2lj\pi }{2i+1})\hat{S}(2m+1,2i+1,j),
\label{snp}
\end{equation}
for $l\allowbreak =\allowbreak 1,\ldots ,i.$
\end{theorem}

\begin{proof}
Is shifted to Section \ref{dowody}.
\end{proof}

As a simple corollary we get some particular values $S(n,k,l)$

\begin{corollary}
i) 
\begin{equation*}
\frac{1}{\pi ^{2m+2}}\hat{S}(2m+2,4,1)\allowbreak =\allowbreak -\frac{\sqrt{2%
}}{\pi ^{2m+2}}(c_{m}(\pi /4)\allowbreak -\allowbreak \allowbreak \zeta
(2m+2)(1-2^{-2m-1})(1-4^{-2m-2})),
\end{equation*}%
$,$ $m\allowbreak =\allowbreak 0,1,2,\ldots ,$ 
\begin{equation*}
\frac{1}{\pi ^{2m+1}}\hat{S}(2m+1,4,1)\allowbreak =\allowbreak -\sqrt{2}(%
\frac{1}{\pi ^{2m+1}}s_{m}(\pi /4)\allowbreak -\allowbreak
(-1)^{m}E_{2m}/(2^{4m+3}(2m)!)),
\end{equation*}
in particular $\hat{S}(2,4,1)/\pi ^{2}\allowbreak =\allowbreak \sqrt{2}/16,$ 
$\hat{S}(3,4,1)/\pi ^{3}\allowbreak =\allowbreak 3\sqrt{2}/128$

ii) 
\begin{eqnarray*}
\frac{2}{\pi ^{2m+2}}S(2m+2,8,1)\allowbreak &=&-\frac{\sqrt{2}}{\pi ^{2m+2}}%
(c_{m}(\pi /4)\allowbreak -\allowbreak \allowbreak \zeta
(2m+2)(1-2^{-2m-1})(1-4^{-2m-2}))\allowbreak +\frac{1}{\pi ^{2m+2}}\zeta
(2m+2)(1-2^{-2m-2}), \\
\frac{2}{\pi ^{2m+2}}S(2m+2,8,3)\allowbreak &=&\frac{\sqrt{2}}{\pi ^{2m+2}}%
(c_{m}(\pi /4)\allowbreak -\allowbreak \allowbreak \zeta
(2m+2)(1-2^{-2m-1})(1-4^{-2m-2}))\allowbreak +\frac{1}{\pi ^{2m+2}}\zeta
(2m+2)(1-2^{-2m-2}).
\end{eqnarray*}%
In particular $S(2,8,1)\allowbreak =\allowbreak \pi ^{2}(1+\sqrt{2}/2)/16,$ $%
S(2,8,3)\allowbreak =\allowbreak \pi ^{2}(1-\sqrt{2}/2)/16,$ $%
S(4,8,1)\allowbreak =\allowbreak \pi ^{4}(1+11\sqrt{2}/16)/192,$ $%
S(2m+2,8,3)\allowbreak =\allowbreak \pi ^{4}(1-11\sqrt{2}/16)/192.$

iii) $S(3,8,1)\allowbreak =\allowbreak \pi ^{3}(1-3\sqrt{2}/4)/32,$ $%
S(3,8,3)\allowbreak =\allowbreak \pi ^{3}(1+3\sqrt{2}/4)/32,$ $%
S(5,8,1)\allowbreak =\allowbreak 5\pi ^{5}(1-57\sqrt{2}/16)/1536,S(5,8,1)%
\allowbreak =\allowbreak 5\pi ^{5}(1+57\sqrt{2}/16)/1536.$

iv) For $m\in \mathbb{N}_{0}$ 
\begin{eqnarray*}
\sum_{j=0}^{m}\binom{2m+2}{2j}B_{2j}(2^{2j-1}-1)/9^{(m-j)+1}
&=&-B_{2m+2}(4^{m}-1+4^{m}/3^{2m+1}), \\
\sum_{j=0}^{m}\binom{2m+2}{2j}B_{2j}(2^{2j-1}-1)2^{2j}\allowbreak
&=&\allowbreak -B_{2m+2}(2^{2m+1}-1)(2^{2m+2}-1).
\end{eqnarray*}

v) For $m\in \mathbb{N}:$%
\begin{equation}
\frac{2(2m+1)!}{\pi ^{2m+1}}S(2m+1,3,1)\allowbreak =\allowbreak (-1)^{m+1}%
\frac{4\sqrt{3}}{3}\sum_{j=0}^{m}\binom{2m+1}{2j+1}%
B_{2(m-j)}(2^{2(m-j)-1}-1)/3^{2j+1}.  \label{S3}
\end{equation}
\end{corollary}

\begin{proof}
i) We apply Theorem \ref{parz}, (\ref{cp}) and (\ref{sp}) with $i\allowbreak
=\allowbreak 2.$ ii) We use (\ref{S1_2}) and (\ref{S*}). iii) We argue in a
similar manner. iv) We set $i\allowbreak =\allowbreak 1$ and $l\allowbreak
=\allowbreak 1$ in (\ref{cn}) and use the fact that $S(2m+2,3,1)\allowbreak
=\allowbreak \frac{\pi ^{2m+2}}{2(2m+2)!}(\frac{2}{3}%
)^{2m+2}(-1)^{m+1}(3^{2m+2}-1)B_{2m+2}.$ Then we set $i\allowbreak
=\allowbreak 1$ and $l\allowbreak =\allowbreak 1$ in (\ref{cp}) and then
cancel out $\pi ^{2m+2}/(2m+2)!.$ v) We set $i\allowbreak =\allowbreak 1,$
and $l\allowbreak =\allowbreak 1$ in (\ref{sn}) .
\end{proof}

Based on the above mentioned theorem we can deduce our main result:

\begin{theorem}
\label{alg}For every $\mathbb{N\ni }k,n\geq 2,$ and $1\allowbreak \leq
\allowbreak l\leq k$ the numbers $S(n,k,l)/\pi ^{n}$ and $\hat{S}(n,k,l)/\pi
^{n}$ are algebraic.
\end{theorem}

\begin{proof}
First of all recall that in the Introdution we have found $S(n,2,1)$ and $%
\hat{S}(n,2,1)$ (compare equations (\ref{sp}) and (\ref{cp})). Further
analyzing equations (\ref{cn})-(\ref{snp}) we notice that quantities $%
S(2m+2,2i+1,j),$ $\hat{S}(2m+2,2i,j),$ $\hat{S}(2m+2,2i+1,j),$ $%
S(2m+1,2i+1,j),$ $\hat{S}(2m+1,2i,j)$ satisfy certain systems of linear
equations with matrices $\allowbreak \allowbreak \lbrack (-1)^{j}\cos \frac{%
(2l-1)j\pi }{n}]_{l,j=1,\ldots ,\left\lfloor (n-1)/2\right\rfloor
},\allowbreak \lbrack (-1)^{j}\sin \frac{(2l-1)j\pi }{n}]_{l,j=1,\ldots
,\left\lfloor (n-1)/2\right\rfloor }$ and $[(-1)^{j}\cos \frac{2lj\pi }{n}%
]_{l,j=1,\ldots ,\left\lfloor (n-1)/2\right\rfloor },\allowbreak $\newline
$[(-1)^{j}\sin \frac{2lj\pi }{n}]_{l,j=1,\ldots ,\left\lfloor
(n-1)/2\right\rfloor }$. These matrices are non-singular since their
determinants differ form the determinants of matrices $C_{n},S_{n},$ $%
C_{n}^{\ast }$ and $S_{n}^{\ast }$ (defined by (\ref{CnSn}) and (\ref{Cn*Sn*}%
), below) only possibly in sign. As shown in Lemma \ref{aux} matrices $%
C_{n},S_{n},$ $C_{n}^{\ast }$ and $S_{n}^{\ast }$ are nonsingular. Moreover
notice that all entries of these matrices are algebraic numbers. On the
other hand as one can see the numbers $c_{m}(l\pi /k)/\pi ^{2m+2}$ and $%
s_{m}(l\pi /k)/\pi ^{2m+1}$ for all $\mathbb{N\ni }m,$ $k,$ $l<\left\lfloor
k/2\right\rfloor $ are rational. Hence solutions of these equations are
algebraic numbers. Now taking into account that for $k\allowbreak
=\allowbreak 2$ we do not get in fact any equation for $\hat{S}(n,2,1)$ .
Thus we deduce that for $n,k,l\in \mathbb{N}$ such that $n\geq 2,k>2,$ $%
l\leq k-1$ the numbers $\hat{S}(n,l,k)/\pi ^{n}$ and for $n\geq 2,$ odd $k$
and $l\leq k-1$ the numbers $S(n,k,l)/\pi ^{n}$ are algebraic.

To show that numbers $S(n,l,k)/\pi ^{n}$ are algebraic also for even $l$ we
refer to (\ref{S1_2}) and (\ref{S*}). First taking $l$ odd we see that $%
S(n,2l,k)/\pi ^{n}$ for all $k\leq l$ (the case $k\allowbreak =\allowbreak l$
leads to $S(n,2,1)$ considered by the end of Introduction) these numbers are
algebraic. Then taking $l\allowbreak =\allowbreak 2m$ with $m$ even we
deduce in the similar way that $S(n,4m,k)/\pi ^{n},$ $k\leq 2m$ are
algebraic. And so on by induction we deduce that indeed for all even $l$ and 
$k\leq l/2$ numbers $S(n,l,k)$ are algebraic.
\end{proof}

As an immediate corollary of Theorem \ref{alg} and (\ref{S1_2}) and (\ref{S*}%
) we have the following result:

\begin{proposition}
For $1<n\in \mathbb{N}$ and $p\in \mathbb{Q\cap (}0,1\mathbb{)}$ $:$ \newline
the numbers $(\zeta (n,p)+(-1)^{n}\zeta (n,1-p))/\pi ^{n}$ and $(\hat{\zeta}%
(n,p)+(-1)^{n-1}\hat{\zeta}(n,1-p))/\pi ^{n}$ are algebraic.
\end{proposition}

\section{ Proofs\label{dowody}}

\begin{proof}[Proof of Lemma \protect\ref{fourier}]
We will need the following observations: For all $m\geq 0,$ $n\geq 1$: 
\begin{eqnarray}
\mathcal{F}_{n}^{s}(x^{2m+1})\allowbreak &=&\allowbreak (-1)^{n+1}\frac{%
2(2m+1)!}{n^{2m+1}}\sum_{k=0}^{m}(-1)^{k}\pi ^{2(m-k)}\frac{n^{2(m-k)}}{%
(2m+1-2k)!},  \label{_1} \\
\mathcal{F}_{n}^{c}(x^{2m}) &=&(-1)^{n}\frac{2(2m)!}{n^{2m}}%
\sum_{k=0}^{m-1}(-1)^{k}\pi ^{2(m-1-k)}\frac{n^{2(m-1-k)}}{(2m-1-2k)!}%
\allowbreak .  \label{_2}
\end{eqnarray}%
Let us denote: $b_{m}\allowbreak =\allowbreak \frac{1}{\pi }\int_{-\pi
}^{\pi }x^{2m+1}\sin (nx)dx$ and $a_{m}\allowbreak =\allowbreak \frac{1}{\pi 
}\int_{-\pi }^{\pi }x^{2m}\cos (nx)dx.$ Integrating by parts twice we obtain
the following recursions 
\begin{equation*}
b_{m}\allowbreak =\allowbreak 2(-1)^{n+1}\frac{\pi ^{2m-1}}{n}\allowbreak
-\allowbreak \frac{(2m+1)2m}{n^{2}}b_{m-1}\allowbreak ,a_{m}\allowbreak
=(-1)^{n}\frac{4m}{n^{2}}\pi ^{2m-1}\allowbreak +\allowbreak \frac{2m(2m-1)}{%
n^{2}}a_{m-1}
\end{equation*}%
with $b_{0}\allowbreak =\allowbreak (-1)^{n+1}\frac{2}{n},$ and $\allowbreak 
$ with $a_{0}\allowbreak =\allowbreak 0\ $for $m\geq 1$. Iterating them we
get (\ref{_1}) and (\ref{_2}).

We show (\ref{wp}) and (\ref{wn}) by straightforward computation. First let
us consider (\ref{wp}). Making use of (\ref{_2}) and after some algebra
including changing the order of summation and setting $l=m-k-j$ we get: 
\begin{equation*}
\mathcal{F}_{n}^{c}(c_{m}(.))\allowbreak =\allowbreak
(-1)^{m+n+1}\sum_{k=0}^{m}\frac{(-1)^{k}n^{-2(k+1)}}{(2m-2k+1)!}\pi
^{2(m-k)}\sum_{l=0}^{m-k}B_{2l}(2^{2l-1}-1)\binom{2m-2k+1}{2l}
\end{equation*}%
Now we use \cite{Szab13}(12) with $x=0$ and \cite{Szab13}(4) to show that 
\begin{equation*}
\sum_{l=0}^{m-k}\binom{2m-2k+1}{2l}B_{2l}(2^{2l-1}-1)\allowbreak
=\allowbreak \left\{ 
\begin{array}{ccc}
0 & if & k<m \\ 
-1/2 & if & k=m%
\end{array}%
\right.
\end{equation*}%
Hence indeed $\mathcal{F}_{n}^{c}(c_{m}(.))\allowbreak =\allowbreak \frac{%
(-1)^{n}}{n^{2m+2}}.$ To get (\ref{wp}) we proceed similarly and after some
algebra we get 
\begin{gather*}
\mathcal{F}_{n}^{s}(s_{m}(.))\allowbreak =\left( -1\right)
^{m+n}\sum_{k=0}^{m}\frac{(-1)^{k}\pi ^{2(m-k)}}{n^{2k+1}(2m-2k+1)!}%
\sum_{l=0}^{m-k}\binom{2m-2k+1}{2l}B_{2l}(2^{2j-1}-1) \\
=\frac{(-1)^{n+1}}{n^{2m+1}}.
\end{gather*}

To get $\frac{1}{\pi }\int_{-\pi }^{\pi }c_{m}(x)dx$ we proceed as follows:

\begin{eqnarray*}
\frac{1}{\pi }\int_{-\pi }^{\pi }c_{m}(x)dx\allowbreak &=&\frac{%
2(-1)^{m+1}\pi ^{2m+2}}{(2m+2)!}\sum_{j=0}^{m}\binom{2m+2}{2(m-j)}\frac{1}{%
(2j+3)}B_{2(m-j)}(2^{2(m-j)-1}-1)\allowbreak \\
&=&\allowbreak \frac{2(-1)^{m+1}\pi ^{2m+2}}{(2m+3)!}\sum_{k=0}^{m}\binom{%
2m+3}{2k}B_{2k}(2^{2k-1}-1)\allowbreak .
\end{eqnarray*}%
$\allowbreak $ Now we again manipulate \cite{Szab13} (4) and (12) to get the
desired formula.
\end{proof}

The proof of Theorem \ref{alg} is based on the Lemma concerning the
following matrices

\begin{gather}
C_{n}=[\cos \frac{(2l-1)j\pi }{n}]_{l,j=1,\ldots ,\left\lfloor
(n-1)/2\right\rfloor },~S_{n}=[\sin \frac{(2l-1)j\pi }{n}]_{l,j=1,\ldots
,\left\lfloor (n-1)/2\right\rfloor },  \label{CnSn} \\
C_{n}^{\ast }=[\cos \frac{2lj\pi }{n}]_{l,j=1,\ldots ,\left\lfloor
(n-1)/2\right\rfloor },~S_{n}^{\ast }=[\sin \frac{2lj\pi }{n}]_{l,j=1,\ldots
,\left\lfloor (n-1)/2\right\rfloor }.  \label{Cn*Sn*}
\end{gather}

\begin{lemma}
$\allowbreak $\label{aux}For every $\mathbb{N\ni }n\geq 2$ matrices $C_{n},$ 
$S_{n},$ $C_{n}^{\ast },$ $S_{n}^{\ast }$ are nonsingular.
\end{lemma}

\begin{proof}
The proof is based on the following observations. Let $i$ denote unitary
unit. Since $\sin x\allowbreak =\allowbreak (\exp (ix)\allowbreak
-\allowbreak \exp (-ix))/2i$ one can easily notice that matrices $S_{n}$ and 
$S_{n}^{\ast }$ are of the type $[x_{l}^{j}-x_{l}^{j}]_{l,j=1,\ldots
,\left\lfloor (n-1)/2\right\rfloor }$ where $x_{l}$ equals respectively $%
\exp (\frac{(2l-1)i\pi }{n})$ and $\exp (\frac{2li\pi }{n})$ ($n$ is odd in
this case). Now we apply formula (2.3) of \cite{kratt99} and deduce that
these matrices are nonsingular provided that quantities $x_{l}$ are nonzero,
are all different, and $x_{l}x_{k}\allowbreak \neq \allowbreak 1,$ $%
k,l\allowbreak =\allowbreak 1,\ldots ,\left\lfloor (n-1)/2\right\rfloor $
which is the case in both situations. Now notice that $\forall n\geq 0:$ $%
\cos (\alpha )\allowbreak =\allowbreak (-1)^{n-1}\sin ((2n-1)\pi /2+\alpha
). $ Hence $\cos \frac{(2l-1)j\pi }{n}\allowbreak =\allowbreak
(-1)^{l-1}\sin ((2l-1)\pi /2+\frac{(2l-1)j\pi }{n})\allowbreak =\allowbreak
(-1)^{l-1}\sin (\frac{\pi (j-1/2)(2l-1)}{n})\allowbreak =\allowbreak
(-1)^{l-1}(x_{l}^{(j-1/2)}\allowbreak -\allowbreak x_{l}^{-(j-1/2)})/(2i)$
and we use formula (2.4) of \cite{kratt99} to deduce that matrix $C_{n}$ is
nonsingular. Now let us consider matrix $C_{n}^{\ast }$ . Notice that then $%
n\allowbreak =\allowbreak 2m+1$ for some natural $m.$ $\cos (\frac{2lj\pi }{%
2m+1})\allowbreak =\allowbreak (-1)^{j}\cos (j\pi -\frac{2lj\pi }{2m+1}%
)\allowbreak =\allowbreak (-1)^{j}\cos (\frac{j\pi (2m-l+1)-1)}{2m+1}).$ And
thus we have reduced this case to the previous one. Hence $C_{n}^{\ast }$ is
also nonsingular.
\end{proof}

\begin{proof}[Proof of Theorem \protect\ref{parz}]
In the proofs of all assertions we use basically the same tricks concerning
properties of trigonometric functions and changing order of summation. That
is why we will present proof of only one (out of 6) in detail. All remaining
proofs are similar. We start with an obvious identity: 
\begin{equation}
c_{m}(\frac{\pi l}{k})\allowbreak =\allowbreak \pi ^{2m+2}(-1)^{m}B_{2m+2}%
\frac{(2^{2m+1}-1)}{(2m+2)!}\allowbreak +\allowbreak \sum_{j=1}^{k}\cos (%
\frac{jl\pi }{k})(-1)^{j}\sum_{r=0}^{\infty }\frac{\left( -1\right) ^{r(k+l)}%
}{(rk+j)^{2m+2}}  \label{1_step}
\end{equation}%
that is true since by $\cos (n\pi +\alpha )\allowbreak =\allowbreak
(-1)^{n}\cos \alpha $ we have: $\allowbreak $%
\begin{equation*}
\sum_{n\geq 1}^{\infty }(-1)^{n}\cos (\frac{n\pi l}{k})/n^{2m+2}\allowbreak
=\allowbreak \sum_{j=1}^{k}\sum_{r=0}^{\infty }\left( -1\right)
^{r(k+l)+j}\cos (\frac{j\pi l}{k})/(rk+j)^{2m+2}\allowbreak ,
\end{equation*}%
$\allowbreak \allowbreak $

Hence we have to consider two cases $k+l$ odd and $k+l$ even.

If $k+l$ is even then we have 
\begin{equation*}
\sum_{n\geq 1}^{\infty }(-1)^{n}\cos (\frac{n\pi l}{k})/n^{2m+2}\allowbreak =%
\frac{1}{k^{2m+2}}\sum_{j=1}^{k}(-1)^{j}\cos (\frac{\pi jl}{k})\zeta
(2m+2,j/k).
\end{equation*}

For $k\allowbreak =\allowbreak 2i+1$ and $l$ odd we have: $%
\sum_{j=1}^{2i+1}(-1)^{j}\cos (j\pi /(2i+1))\zeta (2m+2,j/(2i+1))$. Now
recall that for $j\allowbreak =\allowbreak 2i+1$ we get $\cos (\pi
)\allowbreak =\allowbreak -1$ and that for $j\allowbreak =\allowbreak
1,\ldots ,i$ we have $\cos ((2i+1-j)l\pi /(2i+1)\allowbreak =\allowbreak
-\cos (jl\pi /(2i+1)\allowbreak $ and $(-1)^{2i+1-j}\allowbreak =\allowbreak
-(-1)^{j}.$ Thus we get 
\begin{gather*}
\sum_{j=1}^{2i+1}(-1)^{j}\cos (\frac{lj\pi }{2i+1})\zeta (2m+2,\frac{j}{2i+1}%
)=\zeta (2m+2) \\
+\sum_{j=1}^{i}(-1)^{j}\cos (\frac{jl\pi }{2i+1})(\zeta (2m+2,\frac{j}{2i+1}%
)+\zeta (2m+2,1-\frac{j}{2i+1})).
\end{gather*}%
Now we use (\ref{S1_2}). Since $k+l$ even and $k$ even leads to triviality
we consider two cases leading to $k+l$ odd i.e. $k\allowbreak =\allowbreak
2i $ and $l=2m-1$ odd $m<i$ which leads to:%
\begin{gather*}
\sum_{j=1}^{2i}\cos (\frac{j(2m-1)\pi }{2i})(-1)^{j}\sum_{r=0}^{\infty
}\left( -1\right) ^{r(2i++2m+1)}/(r(2i+1)+j)^{2m+2}\allowbreak \\
=\allowbreak (-1)^{m}B_{2m+2}\frac{\pi ^{2m+2}}{(2m+2)!}\frac{2^{2m+1}-1}{%
2^{2m+2}i^{2m+2}} \\
+\sum_{j=1}^{i-1}(-1)^{j}\cos (\frac{jl\pi }{2i}))(S(2m+2,4i,j)\allowbreak
-\allowbreak S(2m+2,4i,2i-j)).
\end{gather*}%
We used here (\ref{zeta_parz}) (\ref{S1_2}) and (\ref{1_step}).

Similarly we consider the case $k=2i+1$ and $l=2m$ which leads to (\ref{cnp}%
). Note that this case is very similar to the case $k\allowbreak
=\allowbreak 2i+1$ $l\allowbreak =\allowbreak 2m-1.$

Assertions iv)-vi) are proved in the similar way. The only difference is
that we do not have to remember about $\frac{1}{2}\mathcal{F}_{0}^{c}$ and
that we utilize properties of $\sin $ function not of $\cos .$
\end{proof}

\end{document}